\documentclass[11pt, leqno]{amsart}

\usepackage{amssymb, amsmath, color, graphicx, amsbsy, amsfonts}
\usepackage{t1enc}
\usepackage[cp1250]{inputenc}
\usepackage[all]{xy}
\usepackage{colortbl}

\def\cal#1{{\mathcal #1}}

\newtheorem{theorem}{Theorem}[section]
\newtheorem{definition}[theorem]{Definition}
\newtheorem{lemma}[theorem]{Lemma}
\newtheorem{proposition}[theorem]{Proposition}

\newtheorem{remark}[theorem]{Remark}

\newtheorem{ass}[theorem]{}

\numberwithin{equation}{section}

\newcommand{\be}{\begin{equation}}
\newcommand{\ee}{\end{equation}}

\newcommand{\bea}{\begin{eqnarray}}
\newcommand{\eea}{\end{eqnarray}}

\newcommand{\beab}{\begin{eqnarray*}} \newcommand{\eeab}{\end{eqnarray*}}

\def\vp{\varphi}

\def\Lip{\mathrm{Lip}}

\def\ind{\mathrm{ind}\,}

\def\wt{\widetilde}

\def\1{{\bf 1}}
\def\0{{\bf 0}}
\def\R{\mathbb{R}}

\def\la{\langle}
\def\ra{\rangle}

\def\dist{\mathrm{dist}\,}

\def\codim{\mathrm{codim}\,}

\def\wt{\widetilde}

\def\part{\partial}

\def\r{\mathbb{R}}
\def\rn{\mathbb{R}^N}
\def\z{\mathbb{Z}}

\def\eps{\varepsilon}

\def\irn{\int_{\r^N}}

\def\vp{\varphi}
\def\wt{\widetilde}
\def\ol{\overline}

\def\cl{\mathcal{L}}

\def\bbw{\mathbb{W}}

\def\Lip{\mathrm{Lip}}

\title[Bifurcation from infinity]{Bifurcation from infinity for an asymptotically linear Schr\"odinger equation}

\author{Wojciech Kryszewski}
\address{Faculty of Mathematics and Computer Sciences, Nicolaus Copernicus University, Chopina 12/18, 87-100 Toru\'n, Poland}
\email{wkrysz@mat.umk.pl}

\author{Andrzej Szulkin}
\address{Department of Mathematics, Stockholm University, 106 91 Stockholm, Sweden}
\email{andrzejs@math.su.se}



\date{}

\begin{document}

\baselineskip=15pt
\maketitle

\begin{abstract}
We consider the asymptotically linear Schr\"odinger equation \eqref{se} and show that if $\lambda_0$ is an isolated eigenvalue for the linearization at infinity, then under some additional conditions there exists a sequence $(u_n,\lambda_n)$ of solutions such that $\|u_n\|\to\infty$ and $\lambda_n\to\lambda_0$. Our results extend those by Stuart \cite{st3}. We use degree theory if the multiplicity of $\lambda_0$ is odd and Morse theory (or more specifically, Gromoll-Meyer theory) if it is not. 
\end{abstract}

\section{Introduction}

In this paper we consider the Schr\"odinger equation
\begin{equation} \label{se}
-\Delta u + V(x)u = \lambda u + f(x,u), \quad x\in\rn,
\end{equation}
where $\lambda$ is a real parameter, $V\in L^\infty(\rn)$,  $f(x,u)/u\to m(x)$ as $|u|\to\infty$, $m\in L^\infty(\rn)$ and $\lambda_0$ is an isolated eigenvalue of finite multiplicity for $\cl := -\Delta+V(x)-m(x)$. $\cl$ will be considered as an operator in $L^2(\rn)$. It is well known (see e.g.\ \cite{Simon}) that $\cl$ is selfadjoint and its domain $D(\cl)$ is the Sobolev space $H^2(\rn)$.
We shall show that if the distance from $\lambda_0$ to the essential spectrum  $\sigma_e(\cl)$ of $\cl$ is larger than the Lipschitz constant of $f-m$ (with respect to the $u$-variable), then there exists a sequence of solutions $(u_n,\lambda_n)\subset H^2(\rn)\times\r$ such that $\|u_n\|\to\infty$ and $\lambda_n\to\lambda_0$. See Theorems \ref{thm2} and \ref{thm2a} for more precise statements. We shall say that these solutions \emph{bifurcate from infinity} or that $\lambda_0$ is \emph{an asymptotic bifurcation point}. Our results extend those by Stuart \cite{st3} who has shown using degree theory that if $f(x,u)=f(u)+h(x)$, then asymptotic bifurcation occurs if $\lambda_0$ is of odd multiplicity and the bifurcating set contains a continuum.

Both here and in \cite{st3} (see also \cite{st2}) the result is first formulated in terms of an abstract operator equation. Let $E$ be a Hilbert space, $L: D(L)\to E$ a selfadjoint linear operator and let $N: E\to E$ be a continuous nonlinear operator which is asymptotically linear in the sense of Hadamard ($H$-asymptotically linear for short, see Definition \ref{defi1}(i)). We show that if $\lambda_0$ is an isolated eigenvalue of odd multiplicity for $L$ and if the distance $\dist(\lambda_0, \sigma_e(L))$ from $\lambda_0$ to the essential spectrum of $L$ is larger than the asymptotic Lipschitz constant of $N$ (introduced in Definition \ref{defi1}(ii)), then $\lambda_0$ is an asymptotic bifurcation point for the equation
\begin{equation} \label{oe}
Lu = \lambda u + N(u), \quad u\in D(L).
\end{equation}
Here we have assumed for notational simplicity that the asymptotic derivative $N'(\infty)$ of $N$ is 0, see Theorem \ref{thm1} for the full statement. This theorem slightly extends some results in \cite{st2, st3} where the distance condition on $\lambda_0$ was somewhat stronger. If $N$ is the gradient of a $C^1$-functional and $\lambda_0$ is an isolated eigenvalue of finite (not necessarily odd) multiplicity, we show that under an additional hypothesis $\lambda_0$ is an asymptotic bifurcation point for \eqref{oe}. The exact statement is given in Theorem \ref{thm1a}. Existence of asymptotic bifurcation when the multiplicity of $\lambda_0$ is even seems to be new  and is the main abstract result of this paper. A related problem $u = \lambda(Au+N(u))$ has been considered in \cite{dh, to2} under the assumptions that $A$ is bounded linear, $A+N$ is the gradient of a functional and a $k$-set contraction, and $N$ is asymptotically linear \emph{in the stronger sense of Fr\'echet}. It was then shown that each eigenvalue $1/\lambda_0$ of $A$ with $|\lambda_0k|<1$ is an asymptotic bifurcation point. However, the arguments there seem to break down in our case.

The proofs in \cite{st2, st3} were effected by first making the inversion $u\mapsto u/\|u\|^2$ (an idea that goes back to Rabinowitz \cite{ra2} and Toland \cite{to}). In this way the problem is  transformed to that of looking for bifurcation from 0 instead of infinity. In the next step a finite-dimensional reduction is performed and finally it is shown that since $\lambda_0$ has odd multiplicity, the Brouwer degree for the linearization of the reduced operator at $u=0$ changes as $\lambda$ passes through $\lambda_0$. This forces bifurcation, and an additional argument which goes back to \cite{ra1} and uses degree theory in an essential way, shows that there is a continuum bifurcating from $(0,\lambda_0)$. Since the degree does not change if the multiplicity of $\lambda_0$ is even, in Theorem \ref{thm1a} we use Morse theory instead, and therefore we need the assumption that $N$ is the gradient of a functional. Morse theory can only assert that there exists a sequence, and not necessarily a continuum, bifurcating from infinity. Let us also point out that in \cite{st2} a more general operator equation of the form $F(\lambda,u)=0$ has been considered ($F(\lambda,\cdot)$ acts between two Banach spaces). Here we will only be concerned with \eqref{oe}, and this allows some simplifications of Stuart's arguments (in particular in the part involving the finite-dimensional reduction). Since we do not make inversion, we get a less restrictive bound for the distance from $\lambda_0$ to the essential spectrum.

The fact that $\dist(\lambda_0,\sigma_e(L))$ is larger than the Lipschitz constant of $N$ at infinity is needed in order to perform a finite-dimensional reduction of Liapunov-Schmidt type. As we shall see, if the distance condition is satisfied, then one can find an orthogonal decomposition $E=Z\oplus W$, where $\dim Z<\infty$, such that writing $u=z+w\in Z\oplus W$, it is possible to use the contraction mapping principle in order to express $w$ as a function of $z$ and $\lambda$. Although one may think this is only a technical condition, it has been shown by Stuart \cite[Section 5.2]{st3} that there exist examples where asymptotic bifurcation does not occur at eigenvalues of odd multiplicity (and in Section 5.3 there one finds an example where asymptotic bifurcation occurs when $\lambda_0$ is not an eigenvalue). So the above condition, or some other, is needed.

The reason for requiring $N$ to be $H$-asymptotically and not just asymptotically linear (in the sense of Fr\'echet) is that, in contrast to the situation when \eqref{se} is considered for $x$ in a bounded domain, we cannot expect the Nemytskii operator $N$ induced by $f$ to be asymptotically linear. Indeed, it has been shown in \cite{st1} that if $f(u)/u\to m$ as $|u|\to\infty$, then $N$ is always $H$-asymptotically linear, and it is asymptotically linear if and only if $f(u)=mu$. In the proof of Theorem \ref{thm2} we show that also the Nemytskii operator corresponding to $f(x,u)$ is $H$-asymptotically linear if $f(x,u)/u\to m(x)$ as $|u|\to \infty$. The related concept of $H$-differentiability in the context of elliptic equations in $\rn$ has been introduced in a series of papers by Ev\'equoz and Stuart, see e.g. \cite{es}. 

Now we can state our main results. The symbols $N'(\infty)$ and $\Lip_\infty$ (denoting asymptotic $H$-derivative and asymptotic Lipschitz constant) which appear below are introduced in Definition \ref{defi1}. 

\begin{theorem} \label{thm1}
Let $E$ be a Hilbert space and suppose that $L: D(L)\to E$ is a selfadjoint linear operator. Suppose further that \\
(i) $N$ is $H$-asymptotically linear and $N'(\infty): E\to E$ is selfadjoint, \\
(ii) $\lambda_0$ is an isolated eigenvalue of odd multiplicity for $L-N'(\infty)$ and 
\[
\Lip_\infty(N-N'(\infty))< \dist(\lambda_0,\sigma_e(L-N'(\infty))).
\]
Then $\lambda_0$ is an asymptotic bifurcation point for equation \eqref{oe}. Moreover, there exists a continuum bifurcating from infinity at $\lambda_0$.
\end{theorem}

By a continuum bifurcating from infinity at $\lambda_0$ we mean a closed connected set $\Gamma\subset E\times\r$ of solutions of \eqref{oe} which contains a sequence $(u_n,\lambda_n)$ such that $\|u_n\|\to\infty$, $\lambda_n\to\lambda_0$. This theorem should be compared with Theorem 4.2 and Corollary 4.3 in \cite{st3} (see also Theorem 6.3 in \cite{st2}) where the distance condition was somewhat stronger than in (ii) above. 
The main ingredient in the proof is a finite-dimensional reduction which roughly speaking goes as follows. Let $W$ be an $L$-invariant subspace of $E$ such that $\codim W<\infty$ and $Z := W^\bot\subset D(L)$. Let $P: E\to W$ be the orthogonal projection and write $w=Pu$, $z=(I-P)u$. Then \eqref{oe} is equivalent to the system 
\begin{align*}
L w -\lambda w&=PN(w+z),\\
Lz- \lambda z&=(I-P)N(w+z).
\end{align*}
Choosing an appropriate $W$, $\delta>0$ small enough and $R>0$ large enough, one can solve uniquely for $w$ in the first equation provided $|\lambda-\lambda_0|\le\delta$ and $\|z\|\ge R$. In this way we obtain $w=w(\lambda,z)$ which inserted in the second equation gives a (finite-dimensional) problem on $Z\setminus B_R(0)$.
See Proposition \ref{reduction} for more details. Now the proof of Theorem \ref{thm1} is completed by a well-known argument using Brouwer's degree. 

If $N$ is a potential operator, then the reduced problem has variational structure. More precisely, suppose $N(u)=\nabla\psi(u)$ for some $\psi\in C^1(E,\r)$ and let  $\Phi_\lambda(u) := \frac12\la Lu-\lambda u,u\ra -\psi(u)$. Then the functional $\vp_\lambda$ given by $\vp_\lambda(z) = \Phi_\lambda(w(\lambda,z)+z)$ is of class $C^1$ and $z\in Z\setminus \ol B_R(0)$ is a critical point of $\vp_\lambda$ if and only if $u=w(\lambda,z)+z$ is a solution of \eqref{oe}, see Proposition \ref{varred}. Recall that a functional $\vp$ is said to satisfy the Palais-Smale condition ((PS) for short) if each sequence $(z_n)$ such that $\vp(z_n)$ is bounded and $\vp'(z_n)\to 0$ contains a convergent subsequence.

\begin{theorem} \label{thm1a}
Let $E$ be a Hilbert space and suppose that $L: D(L)\to E$ is a selfadjoint linear operator. Suppose further that \\
(i) $N$ is a potential operator, i.e.\ there exists a functional $\psi\in C^1(E,\r)$ such that $\nabla\psi(u)=N(u)$ for all $u\in E$, \\
(ii) $N$ is $H$-asymptotically linear and $N'(\infty): E\to E$ is selfadjoint, \\
(iii) $\lambda_0$ is an isolated eigenvalue of finite multiplicity for $L-N'(\infty)$ and 
\[
\Lip_\infty(N-N'(\infty))< \dist(\lambda_0,\sigma_e(L-N'(\infty))).
\]
If $\vp_{\lambda_0}$ satisfies (PS), then $\lambda_0$ is an asymptotic bifurcation point for equation \eqref{oe}.
\end{theorem}

Note that here we do not assume $\lambda_0$ is of odd multiplicity. In Theorem \ref{thm2a} below we shall give sufficient conditions for $f$ in order that such $\lambda_0$ be an asymptotic bifurcation point for \eqref{se}.

To formulate our results for equation \eqref{se} we introduce the following assumptions on $f$:  
\begin{itemize}
\item[$(f_1)$] $f:\R^N\times\R\to\R$ satisfies the Carath\'eodory condition, i.e., it is continuous in $s$ for almost all $x\in\rn$ and measurable in $x$ for all $s\in\r$, and there exist $\alpha\in L^2(\R^N)$, $\beta\in\r^+$ such that
$|f(x,s)|\leq\alpha(x)+\beta|s|$ for all $x\in\R^N$, $s\in\R$;
\item[$(f_2)$] $f$ is Lipschitz continuous in the second variable, with Lipschitz constant $\Lip(f) := \inf \{C: |f(x,s)-f(x,t)| \le C|s-t| \text{ for all } x\in\rn,\ s,t\in\r\}$; 
\item[$(f_3)$] $\lim_{|s|\to\infty}f(x,s)/s = m(x)$, where $m\in L^\infty(\rn)$;
\item[$(f_4)$] $g(x,s):=f(x,s)-m(x)s$ is bounded by a constant independent of $x\in\R^N$ and $s\in\R$;
\item[$(f_5)$] Assume the limits $g_\pm(x):=\lim_{s\to\pm\infty}g(x,s)$ exist and either $\pm g_\pm\ge 0$  a.e. or  $\pm g_\pm\le 0$ a.e. In addition, there exists a set of positive measure on which none of $g_\pm$ vanishes;
\item[$(f_6)$] Assume the limits $h_\pm(x):=$ $\lim_{s\to\pm\infty}g(x,s)s$ exist, $h_\pm\in L^\infty(\rn)$ and either $g(x,s)s\ge 0$ or $g(x,s)s\le 0$ for all $x\in\rn$, $s\in\r$. In addition, there exists a set of positive measure on which none of $h_\pm$ vanishes.
\end{itemize}

Note that if $f(x,s) = \alpha(x)+f_0(s)$ and $|f_0(s)| \le \beta|s|$, where $\alpha\in L^2(\rn)$, $\beta>0$ and $f_0$ is continuous, then $f$ satisfies $(f_1)$. As we have already mentioned, such functions $f$ have been considered in \cite{st3}. 

\begin{theorem} \label{thm2}
Suppose that $V\in L^\infty(\rn)$ and $f$ satisfies $(f_1)$-$(f_3)$. Let $g(x,s):=f(x,s)-m(x)s$. If $\lambda_0$ is an isolated eigenvalue of odd multiplicity for $-\Delta+V-m$ and $\Lip(g) < \dist(\lambda_0, \sigma_e(-\Delta+V-m))$, then $\lambda_0$ is an asymptotic bifurcation point for equation \eqref{se}. Moreover, there exists a continuum bifurcating from infinity at $\lambda_0$.
\end{theorem}

This strengthens some of the results of \cite[Theorem 5.2]{st3}. Using examples in \cite[Theorems 5.4, 5.6]{st3} and the remarks following them we shall show in Remark \ref{nobif} that the condition on $\Lip(g)$ above is sharp in the sense that if $\Lip(g) > \dist(\lambda_0, \sigma_e(-\Delta+V-m))$, then there may be no bifurcation at a simple eigenvalue.

\begin{theorem} \label{thm2a}
Suppose that $V\in L^\infty(\rn)$ and $f$ satisfies $(f_1)$-$(f_4)$ and either $(f_5)$ or $(f_6)$.  If $\lambda_0$ is an isolated eigenvalue of finite multiplicity for $-\Delta+V-m$ and $\Lip(g) < \dist(\lambda_0, \sigma_e(-\Delta+V-m))$,
then $\lambda_0$ is an asymptotic bifurcation point for equation \eqref{se}. 
\end{theorem}

To our knowledge there are no earlier results on asymptotic bifurcation for \eqref{se} if the multiplicity of $\lambda_0$ is even.

\medskip

The rest of the paper is organized as follows. Section \ref{sec:prel} contains some preliminary material. In Section \ref{sec:red} a finite-dimensional reduction is performed. In Section \ref{sec:proof1} we prove Theorems \ref{thm1} and \ref{thm1a}, and Section \ref{sec:proof2} is concerned with the proofs of Theorems \ref{thm2} and \ref{thm2a}.  

\medskip

\noindent\textbf{Notation.} $\la\cdot\,,\cdot\ra$ denotes the inner product in a (real) Hilbert space $E$ and $\|\cdot\|$ is the corresponding norm. If $\Phi\in C^1(E,\r)$, then $\Phi'(u)\in E^*$ is the Fr\'echet derivative of $\Phi$ at $u$ and $\nabla\Phi(u)$ (the gradient of $\Phi$ at $u$) is the corresponding element in $E$, i.e., $\la\nabla\Phi(u),v\ra = \Phi'(u)v$. The graph norm corresponding to a linear operator $L$ will be denoted by $\|\cdot\|_L$. The symbol $B_r(a)$ will stand for the open ball centered at $a$ and having radius $r$, and we denote the $L^p$-norm of $u$ by $\|u\|_p$.

\section{Preliminaries} \label{sec:prel}

Let $X, Y$ be (real) Banach spaces and let $N:X\setminus B_R(0)\to Y$.

\begin{definition}\label{defi1}
{\em (i) We say that $N$ is \emph{asymptotically linear in the sense of Hadamard} ({\em  H-asympto\-tically linear} for short) if there is a bounded linear operator $B:X\to Y$ such that
$$\lim_{n\to\infty}\frac{N(t_nu_n)}{t_n}=Bu$$
for all sequences $(t_n)\subset\R$, $(u_n)\subset X$ such that  $u_n\to u$ and $\|t_nu_n\|\to\infty$. The operator $B$ is called the \emph{asymptotic $H$-derivative }and is denoted by $N'(\infty)$.\\
\indent (ii) We say that $N$ is \emph{Lipschitz continuous at infinity} if
$$\Lip_\infty(N):=\lim_{R\to\infty}\sup\left\{\frac{\|N(u)-N(v)\|}{\|u-v\|}: u\neq v,\;\|u\|,\|v\|\geq R\right\}<\infty.$$
Note that the limit is well defined because the supremum above decreases as $R$ increases.}\end{definition}

\begin{remark}\label{as lin} 
{\em (i) The definition of $H$-asymptotic linearity given in \cite{st1} is in fact a little different but the one formulated above is somewhat more convenient and is equivalent to the original one as has been shown in \cite[Theorem A.1]{st1}.\\
\indent (ii) Recall that $N$ is {\em asymptotically linear} (in the sense of Fr\'echet) if there is a bounded linear operator $B$ such that
\begin{equation} \label{*}
\lim_{\|u\|\to \infty}\frac{\|N(u)-Bu\|}{\|u\|}=0.
\end{equation}
It is clear that if $N$ is asymptotically linear, then it is $H$-asymptotically linear and $N'(\infty)=B$. If, however, $\dim X<\infty$, then $H$-asymptotic linearity is equivalent to asymptotic linearity and \eqref{*} above holds for $B=N'(\infty)$, see \cite[Remark 2]{st1}.
}\end{remark}

Recall that a linear operator
$L :D(L)\subset X\to Y$ is called {\em a Fredholm operator} if it is densely defined, closed, $\dim N(L)<\infty$ (where $N(L)$ is the kernel of $L$),
the range $R(L)$ is closed and $\codim R(L)<\infty$. The number
$$ \ind(L):=\dim N(L)-\codim R(L)$$
is the {\em index of $L$} (cf.\ \cite[Section 1.3]{MS}).

Suppose that  $E$ is a real Hilbert space and let $L:D(L)\subset E\to E$ be a selfadjoint Fredholm operator. Then $\ind(L)=0$, $E=N(L)\oplus R(L)$ (orthogonal sum) and $S:=L|_{R(L)\cap D(L)}$ is invertible with bounded inverse. Hence, in view of \cite[Problem III.6.16]{Kato},
$$\|S^{-1}\|=r(S^{-1})=\frac{1}{\dist(0,\sigma(S))}=\frac{1}{\dist(0,\sigma(L)\setminus\{0\})},$$ where $r(S^{-1})$ denotes the spectral radius of
$S^{-1}$. The first equality holds since $S^{-1}$ is selfadjoint, see \cite[(V.2.4)]{Kato}. Recall that a selfadjoint operator is necessarily densely defined and closed.

It is clear that if $W$ is a closed subspace of
$R(L)$, invariant with respect to $L$ (i.e. $L(W\cap D(L))\subset
W$), then $L_W:=L|_{W\cap D(L)}$ is also invertible and
$$\|L_W^{-1}\|=\frac{1}{\dist(0,\sigma(L_W))}.$$

\begin{remark}\label{graph norm} {\em Keeping the above notation observe that $L_W^{-1}:W\to W\cap D(L)$ is bounded with respect to the graph norm $\|\cdot\|_L$ in $W\cap D(L)$ (recall that $\| u\|_L:=\|u\|+\|Lu\|$ for $u\in D(L)$). In fact,
$$\| L^{-1}_Ww\|_L=\|L_W^{-1}w\|+\|w\|\leq \left(1+\frac{1}{\dist(0,\sigma(L_W))}\right)\|w\|,\;\; w\in W.$$
}\end{remark}

\begin{definition}\label{defi gamma} {\em For a selfadjoint Fredholm operator $L:D(L)\to E$, let us put
\be\label{def gamma}\gamma(L):=\inf\{\|(L|_{W\cap D(L)})^{-1}\|:W\in {\cal W}\},\ee
where ${\cal W}$ denotes the family of closed $L$-invariant  linear
subspaces of $R(L)$ such that $\codim W<\infty$ and $W^\bot\subset
D(L)$.}
\end{definition}

\begin{definition}\label{defi2} {\em By the {\em essential spectrum} $\sigma_e(L)$ of a selfadjoint linear operator $L:E\supset D(L)\to E$ we understand the set
\[
\{\lambda\in \mathbb{C}: L-\lambda I \text{ is not a Fredholm operator}\}
\]
(see \cite[\S 1.4]{MS}).
}
\end{definition}

It follows immediately from this definition that $\sigma_e(L)\subset\sigma(L)$ and $\sigma(L)\setminus\sigma_e(L)$ consists of isolated eigenvalues of finite multiplicity.

\begin{theorem}\label{gamma} Let $L:E\supset D(L)\to E$ be a selfadjoint linear
operator and let $\lambda_0\in\sigma(L)\setminus \sigma_e(L)$.
Then $L-\lambda_0 I$ is a Fredholm operator and
$$\gamma(L-\lambda_0 I)=\frac{1}{\dist(\lambda_0,\sigma_e(L))}.$$
If $\sigma_e(L)=\emptyset$ (this is the case e.g. if $L$ is
resolvent compact), then 
$\gamma(L-\lambda_0 I)=0$.
\end{theorem}

\begin{proof} Since
$\sigma_e(L)-\lambda_0=\sigma_e(L-\lambda_0 I)$
and hence
$$\dist(\lambda_0,\sigma_e(L))=\dist(0,\sigma_e(L-\lambda_0I)),$$
we may assume without loss of generality that $\lambda_0=0$ and we will show that
$$\gamma(L)=\frac{1}{\dist(0,\sigma_e(L))}.$$
If $W\in {\cal W}$ and $Z:=W^\bot$, then $\dim Z<\infty$ and $Z\subset D(L)$ is $L$-invariant. Hence $\sigma(L)=\sigma(L|_{W\cap D(L)})\cup \sigma(L|_Z)$. Obviously, any $\lambda\in \sigma(L|_Z)$ is an isolated eigenvalue of  finite multiplicity; thus $\sigma_e(L)\subset \sigma(L|_{W\cap D(L)})$. This implies that
$$\|(L|_{W\cap D(L)})^{-1}\|=\frac{1}{\dist(0,\sigma(L|_{W\cap D(L)})}\geq\frac{1}{\dist(0,\sigma_e(L))}\text{ and therefore }\gamma(L)\geq \frac{1}{\dist(0,\sigma_e(L))}.$$

Take any $0<d<\dist(0,\sigma_e(L))$
and let
$$D=[-d,d]\cap \sigma(L),\;\;B:=\sigma(L)\setminus D.$$
Clearly $D$ is finite:  if $\lambda\in D$, then
$\lambda\in\sigma(L)\setminus \sigma_e(L)$, i.e., $\lambda$ is an
isolated eigenvalue of finite multiplicity. Therefore $B$ is closed and $\sigma_e(L)\subset B$. Obviously, $\sigma(L)=D\cup B$.
Let $Z$ be the subspace spanned by the eigenfunctions corresponding to the eigenvalues in $D$ and let $W=Z^\bot$. Then $Z\subset D(L)$, $W\subset R(L)$,  $Z,W$ are invariant with respect to $L$, $L|_{Z}$ is bounded,
$D=\sigma(L|_Z)$ and $B=\sigma(L|_{W\cap D(L)})$. Clearly, $W\in {\cal W}$ since $\dim Z<\infty$. Now
$$\|(L|_{W\cap D(L)})^{-1}\|=r((L|_{W\cap
D(L)})^{-1})=\frac{1}{\dist(0,\sigma(L|_{W\cap
D(L)}))}=\frac{1}{\dist(0,B)}\leq \frac{1}{d}.$$ This implies the
assertion. Note that if $\sigma_e(L)=\emptyset$, we can choose any $d>0$. Hence $\gamma(L)=0$.
\end{proof}

\begin{remark}\label{rem gamma} 
{\em 
Let $L$ be a Fredholm operator of index 0 and let ${\cal P}(L)$ denote the
collection of all bounded operators $K$ of finite rank and such that $L+K$ is invertible. Clearly,
${\cal P}(L)\neq\emptyset$. Put
$$\wt\gamma(L):=\inf\{\|(L+K)^{-1}\|: K\in {\cal P}(L)\}.$$
Then $\wt\gamma(L)$ corresponds to the notion of {\em essential conditioning
number} in \cite[Section 5.1]{st2}, see also \cite[Section 3.1]{st3} where the definition above appears explicitly.

We claim that if $L$ is a selfadjoint Fredholm operator,  then $\wt\gamma(L) = \gamma(L)$. For $K\in {\cal P}(L)$, $\sigma_e(L)=\sigma_e(L+K)\subset\sigma(L+K)$, hence
$$
\|(L+K)^{-1}\|\geq r((L+K)^{-1})=\frac{1}{\dist(0,\sigma(L+K))} \ge \frac{1}{\dist(0,\sigma_e(L))}.
$$
So $\wt\gamma(L)\ge \gamma(L)$ according to the definition of $\wt\gamma$ and Theorem \ref{gamma}. 
On the other hand, take any $W\in {\cal W}$ and let $Z:=W^\bot$. As before, write $u=z+w\in Z\oplus W$ and let $Ku:=\alpha z-Lz$,
where
$$\alpha:=\inf\{\|Lw\|:w\in W\cap D(L),\,\|w\|=1\} > 0.$$ 
Then $K$ has finite rank and, for $u\in
D(L)$,   $Lu+Ku=Lw+\alpha z$. Hence $L+K$ is invertible and it is easy to see
that
$$ \inf\{\|Lu+Ku\|: u\in D(L),\, \|u\|=1\}\geq \alpha.$$
So
$$
\wt\gamma(L) \le\|(L+K)^{-1}\| \leq \frac{1}{\alpha}=\|(L|_{W\cap
D(L)})^{-1}\|
$$
and $\wt\gamma(L)\leq\gamma(L)$.
We have shown that $\wt\gamma(L) = \gamma(L)$. Therefore Theorem \ref{gamma} may be considered as a refinement of
\cite[Theorem 5.5 and Corollary 5.6]{st2}.
}
\end{remark}

\section{The problem and finite-dimensional reduction} \label{sec:red}

Let $E$ be a real Hilbert space and $L:E\supset D(L)\to E$ a selfadjoint operator.
 We shall study the existence of solutions to the eigenvalue problem \eqref{oe}, i.e.,
\[
Lu=\lambda u+N(u), \quad u\in D(L),\ \lambda\in\R,
\]
or, more precisely, the existence of asymptotic bifurcation of solutions to (\ref{oe}). Recall that $\lambda_0\in\R$ is an {\em asymptotic bifurcation point} for (\ref{oe}) if there exist sequences $\lambda_n\to\lambda_0$ and $(u_n)\subset D(L)$ such that $\|u_n\|\to\infty$ and $Lu_n-N(u_n)=\lambda_nu_n$.

By $X$ we denote the domain $D(L)$ furnished with the graph norm
$$\| u\|_L:=\|u\|+\|Lu\|,\ u\in D(L).$$
Then $X$ is a Banach space, $L$ is bounded as an operator from $X$ to $E$ and the inclusion $i:X\hookrightarrow E$ is continuous.

If $N$ is a potential operator, i.e. there exists $\psi\in C^1(E,\R)$ such that $N=\nabla\psi$, then
along with \eqref{oe} we can consider the existence of critical points of the functional $\Phi_\lambda: X\to\r$, $\lambda\in\R$, given by
$$\Phi_\lambda(u):=\frac{1}{2}\la Lu-\lambda u,u\ra-\psi(u),\quad u\in X.$$
Since $|\la Lu,u\ra|\leq\|Lu\|\|u\|\leq \| u\|_L^2$, $\Phi_\lambda\in C^1(X,\R)$ and
\be\label{wzor pochodnej}\Phi_\lambda'(u)v=\la Lu-\lambda u,v\ra-\la N(u),v\ra,\quad u,v\in X.\ee
It is clear that if $u\in X$  solves (\ref{oe}) for some $\lambda\in\R$, then  $\Phi_\lambda'(u)v=0$ for all $v\in X$, i.e., $u$ is a critical point of $\Phi_\lambda$. Conversely, if $u\in X$ and $\Phi_\lambda'(u)=0$, then $u$ solves (\ref{oe}) since $D(L)$ is dense in $E$. Note that if $L$ is unbounded, then $\Phi_\lambda$ is defined on $D(L)$ and is not $C^1$ with respect to the original norm $\|\cdot\|$ of $E$ on $D(L)$. 

\medskip

In what follows we assume:
\begin{ass}\label{as1}{\em $N$ is $H$-asymptotically linear with $N'(\infty)= 0$;}\end{ass}
\begin{ass}\label{as2} {\em $N$ is Lipschitz continuous at infinity;}\end{ass}
\begin{ass}\label{as3} {\em $\lambda_0=0\in \sigma(L)\setminus \sigma_e(L)$ and $\Lip_\infty(N)<\dist(0,\sigma_e(L))$.}\end{ass}
Observe that these assumptions cause no loss of generality in Theorems \ref{thm1} and \ref{thm1a} since if $N'(\infty)\ne 0$ is selfadjoint and $\lambda_0\ne 0$, then we may replace $L$ by $L-N'(\infty)-\lambda_0I$ and $N$ by $N-N'(\infty)$.

As a first step towards showing that $\lambda_0=0$ is an asymptotic bifurcation point for \eqref{oe} we
perform a kind of a Liapunov-Schmidt finite-dimensional reduction near infinity. Put
$$L_\lambda u:=Lu-\lambda u, \text{ where }u\in D(L_\lambda)=D(L),\ \lambda\in\R$$
and note that the norms $\|\cdot\|_L$ and $\|\cdot\|_{L_\lambda}$ are equivalent.  Given $W\in {\cal W}$, let $P:E\to W$ be the orthogonal projection and $Z:=W^\bot$. Observe that $u=w+z \in D(L)$, where $w\in W$, $z\in Z$, solves (\ref{oe}) if and only if
\begin{align}
\label{P}L_\lambda w & = PN(w+z),\\
\label{I-P} L_\lambda
z & = (I-P)N(w+z).
\end{align}

\begin{proposition}\label{reduction}
There are a subspace $W\in {\cal W}$, numbers $\delta\in (0,\dist(0,\sigma(L)\setminus\{0\})$, $R>0$ and a continuous map $w:[-\delta,\delta]\times (Z\setminus B_R(0))\to W\cap D(L)$ such that \eqref{P} holds for $w=w(\lambda,z)$ and: \\
{\em (i)} For any $\lambda$ with $|\lambda|\le \delta$ , $z,z'\in Z\setminus B_R(0)$ and some constant $c>0$,
\begin{equation}
\label{lip graph}\|w(\lambda,z)-w(\lambda,z')\| \le \|w(\lambda,z)-w(\lambda,z')\|_{L}\leq c\|z-z'\|.
\end{equation}
In particular, $w(\cdot,\cdot)$ is continuous with respect to the graph norm. \\
{\em (ii)}
$w(\lambda,\cdot)$ is $H$-asymptotically linear with $w'(\lambda,\infty) = 0$. \\
\emph{(iii)} $z\in Z\setminus B_R(0)$ is a solution of \eqref{I-P} with $w=w(\lambda,z)$ if and only if $u=w(\lambda,z)+z$ is a solution of \eqref{oe}.
\end{proposition}

Note that the condition on $\delta$ implies invertibility of $L_\lambda$ for $0<|\lambda|\le\delta$. 

\begin{proof} 
(i) According to Definition \ref{defi gamma}  of $\gamma(L)$, Theorem \ref{gamma} and assumption \ref{as3}, there is a closed subspace $W\in {\cal W}$ for which
$$\Lip_\infty(N)\|(L|_{W\cap D(L)})^{-1}\|<1.$$
Hence we can find $\delta \in (0,\dist(0,\sigma(L)\setminus\{0\})$ and $R>0$ such that
$$
k:=\sup_{|\lambda|\le\delta}\|(L_\lambda|_{W\cap D(L)})^{-1}\|\cdot\beta<1,
$$ 
where
\be \label{beta}
\beta:=\sup\left\{\frac{\|N(u)-N(v)\|}{\|u-v\|}: u\neq v,\;\|u\|,\|v\|\geq R\right\}.
\ee
Let $Z := W^\bot$ and let $P: E\to W$ be the orthogonal projection. To facilitate the notation let us put
$$M_\lambda(w+z):=(L_\lambda|_{W\cap D(L)})^{-1}PN(w+z)\in W\cap D(L),\;\;w\in W,\;z\in Z\;\;\hbox{and}\;\;|\lambda|\le\delta.$$
Then \eqref{P} is equivalent to the fixed point equation
\be \label{fix}
w = M_\lambda(w+z).
\ee
Fix $\lambda\in [-\delta,\delta]$ and  $z\in Z$, $\|z\|\geq R$. If $w, w'\in W$, then $\|w+z\|,\|w'+z\|\geq \|z\|\geq R$, so taking into account that $\|P\|=1$, we have
$$\|M_\lambda(w+z)-M_\lambda(w'+z)\|\leq k\|w-w'\|.$$
By the Banach contraction principle there is a unique $w=w(\lambda,z)\in W\cap D(L)$, continuously depending on $\lambda$ and $z$, such that \eqref{fix}, and hence \eqref{P}, holds.
Moreover,
\begin{gather*}
\|w(\lambda,z)-w(\lambda,z')\|=\|M_\lambda(w(\lambda,z)+z)-M_\lambda(w(\lambda,z')+z')\|\leq\\
k\|w(\lambda,z)-w(\lambda,z')+z-z'\|\leq
k\|w(\lambda,z)-w(\lambda,z')\|+k\|z-z'\|
\end{gather*}
for all $|\lambda|\le\delta$, $z,z'\in Z\setminus B_R(0)$. So $\|w(\lambda,z)-w(\lambda,z')\| \le k(1-k)^{-1} \|z-z'\|$. Using this, \eqref{beta} and arguing as above, we obtain
\begin{gather*}
\|L_\lambda w(\lambda,z)-L_\lambda w(\lambda,z')\|= \| PN(w(\lambda,z)+z)-PN(w(\lambda,z')+z')\| \\
\le \beta\|w(\lambda,z)-w(\lambda,z')\| +\beta\|z-z'\| \leq \frac{\beta}{1-k}\|z-z'\|.
\end{gather*}
Since $\|\cdot\|_L$ and $\|\cdot\|_{L_\lambda}$ are equivalent norms, the second inequality in \eqref{lip graph} follows (the first one is obvious).

(ii) To show the $H$-asymptotic linearity of $w(\lambda,\cdot)$ with $w'(\lambda,\infty)=0$, let $(z_n)\subset Z$ and $(t_n)\subset\R$ be sequences such that $z_n\to z$ and $\|t_nz_n\|\to\infty$. Then, for sufficiently large $n$,
$\|w(\lambda,t_nz_n)+t_nz_n\|\geq \|t_nz_n\|\geq R$ and
$$
\|w(\lambda,t_nz_n)\|\leq\|M_\lambda(w(\lambda,t_nz_n)+t_nz_n)-M_\lambda(t_nz_n)\|+
\|M_\lambda(t_nz_n)\|
\leq k\|w(\lambda,t_nz_n)\|+\|M_\lambda(t_nz_n)\|.
$$   Thus, in view of assumption \ref{as1},
\be\label{w as lin}\frac{\|w(\lambda,t_nz_n)\|}{|t_n|}\leq
\frac{1}{1-k}\frac{\|M_\lambda(t_nz_n)\|}{|t_n|}\to 0.\ee 

(iii) is an immediate consequence of (i).
\end{proof}

\begin{remark}\label{pot} {\em Suppose that $z_n\to z$ in $Z$ and take a sequence $(t_n)\subset\R$ such that $\|t_nz_n\|\to\infty$. Then, again in view of the $H$-asymptotic linearity of $N$ and (\ref{w as lin}), we have
\be\label{5}
\frac{N(w(\lambda,t_nz_n)+t_nz_n)}{t_n}
=\frac{N\left(t_n\left(\frac{w(\lambda,t_nz_n)}{t_n}+z_n\right)\right)}{t_n}\to 0
\ee
for each fixed $\lambda\in[-\delta,\delta]$. }
 \end{remark}

If $N=\nabla\psi$, then we let
\be \label{redfcl}
\vp_\lambda(z):=\Phi_\lambda(w(\lambda,z)+z), \quad |\lambda|\le\delta,\ z\in Z\setminus \ol B_R(0).
\ee

\begin{proposition} \label{varred} 
Let $|\lambda|\le\delta$. Then
$\vp_\lambda\in C^1(Z\setminus \ol B_R(0),\R)$ and
\be\label{poch vp}\nabla\vp_\lambda(z)=L_\lambda z-(I-P)N(w(\lambda,z)+z).\ee
Therefore $z\in Z\setminus \ol B_R(0)$ is a critical point of $\vp_\lambda$ if and only if $u=w(\lambda,z)+z$ solves (\ref{oe}). Moreover,  $\nabla\vp_\lambda$ is asymptotically linear with $(\nabla\vp_\lambda)'(\infty)=L_\lambda|_Z$.
\end{proposition}

\begin{proof}
To show (\ref{poch vp}) we shall compute the derivative of $\vp_\lambda$ in the direction $h\in Z$, $h\neq 0$. 
For notational convenience we write $w(z)$ for $w(\lambda,z)$. Let $t>0$,
$$u:=w(z)+z \quad \text{and} \quad \xi:=w(z+th)-w(z)+th.$$
Then we have
$$\vp_\lambda(z+th)-\vp_\lambda(z)=\Phi_\lambda(u+\xi)-\Phi_\lambda(u)-\Phi_\lambda'(u)\xi+
\Phi_\lambda'(u)\xi.$$ 
Clearly, $\xi\neq 0$ as $t>0$. In view of \eqref{wzor pochodnej},  \eqref{P} and since $w(z+th)-w(z)\in W$, 
\begin{align*}
\Phi_\lambda'(u)\xi &  = \la L_\lambda u-N(u),\xi\ra = \la L_\lambda w(z)-PN(u),\xi\ra + \la L_\lambda z-(I-P)N(u),\xi\ra \\
& = \la L_\lambda z-N(u), th\ra  = t\Phi_\lambda'(u)h.
\end{align*}
Hence
\be\label{8}
\frac{\vp_\lambda(z+th)-\vp_\lambda(z)}{t}=
\Phi_\lambda'(u)h+\frac{\|\xi\|_{L}}{t}\cdot\frac{\Phi_\lambda(u+\xi)-\Phi_\lambda(u)-\Phi_\lambda'(u)\xi}{\|\xi\|_{L}}.
\ee
It follows from (\ref{lip graph}) that
$$
\|\xi\|_{L} \le td\|h\|
$$
for some $d>0$. This, together with the Fr\'echet differentiability of $\Phi_\lambda$ on $X$ (i.e., on $D(L)$ with the graph norm) implies that the second term on the right-hand side of (\ref{8}) tends to 0 as $t\to 0$. So
\[
\lim_{t\to 0^+}\frac{\vp_\lambda(z+th)-\vp_\lambda(z)}{t}=\Phi_\lambda'(u)h =
\la L_\lambda z,h\ra-\la(I-P)N(w(z)+z),h\ra.
\]
Therefore $\vp_\lambda$ is continuously G\^ateaux differentiable, hence continuously Fr\'echet differentiable as well, and the derivative is as claimed. 

 If $z\in Z\setminus\ol B_R(0)$ is a critical point of $\vp_\lambda$, then (\ref{I-P}) with $w=w(\lambda,z)$ is satisfied; this together with (\ref{P}) shows that $u=w(\lambda,z)+z$ solves (\ref{oe}).

 Since $\dim Z<\infty$, in order to prove the last part of the assertion it suffices to  show that $\nabla\vp_\lambda$ is $H$-asymptotically linear (see Remark \ref{as lin}(ii)). If $z_n\to z$ in $Z$, $(t_n)\subset\R$ and $\|t_nz_n\|\to\infty$, then, in view of (\ref{5}),
$$\frac{\nabla\vp_\lambda(t_nz_n)}{t_n}=L_\lambda z_n-\frac{(I-P)N(w(t_nz_n)+t_nz_n)}{t_n}\to L_\lambda z.$$
This concludes the proof. 
\end{proof}

\begin{remark} \label{remphi}
{\em (i) Using \eqref{lip graph} and the fact that $\beta$ in \eqref{beta} is finite, it is easy to see that $\nabla\vp_\lambda$ is Lipschitz continuous on $Z\setminus \ol B_R(0)$ and the Lipschitz constant may be chosen independently of $\lambda\in [-\delta, \delta]$.

(ii) In what follows we may (and will need to) assume that $\vp_\lambda$ is defined on $Z$ and not only on $Z\setminus \ol B_R(0)$. Such an extension of $\vp_\lambda$ can be achieved e.g. as follows. Let $\chi\in C^\infty(\r,[0,1])$ be a cutoff function such that $\chi(t)=0$ for $t\le R+1$ and $\chi(t)=1$ for $t\ge R+2$. Set $\wt\vp_\lambda(z) := \chi(\|z\|)\vp_\lambda(z)$. Then $\wt\vp_\lambda$ is of class $C^1$, Lipschitz continuous and $\wt\vp_\lambda(z) = \vp_\lambda(z)$ for $\|z\|>R+2$. In particular, $z\in Z\setminus \ol B_{\wt R}(0)$, where $\wt R:=R+2$, is a critical point of $\wt\vp_\lambda$ if and only if $u=w(\lambda,z)+z$ solves (\ref{oe}).
}
\end{remark}

\section{Proofs of Theorems \ref{thm1} and \ref{thm1a}} \label{sec:proof1} 

In the proof of Theorem \ref{thm1} we shall need the following version of Whyburn's lemma which may be found in \cite[Proposition 5]{Al}:

\begin{lemma}  \label{whyburn}
Let $Y$ be a compact space and $A,B\subset Y$ closed sets. If there is no connected set $\Gamma\subset Y\setminus (A\cup B)$ such that $\ol \Gamma\cap A\neq\emptyset$ and $\ol \Gamma\cap B\neq\emptyset$ ($\ol\Gamma$ stands for the closure of $\Gamma$ in $Y$), then $A$ and $B$ are separated, i.e. there are open sets $U,V\subset Y$ such that $A\subset U$, $B\subset V$, $U\cap V=\emptyset$ and $Y=U\cup V$ (clearly, $U,V$ are closed as well).
\end{lemma}

\begin{proof}[Proof of Theorem \ref{thm1}]
By Proposition \ref{reduction}, it suffices to consider equation \eqref{I-P} with $w=w(\lambda,z)$ which we re-write in the form
\be \label{redeq}
F_\lambda(z) := L_\lambda z - (I-P)N(w(\lambda,z)+z) = 0.
\ee
As in  assumptions \ref{as1}--\ref{as3}, it causes no loss of generality to take $\lambda_0=0$ and $N'(\infty)=0$. Although $F_\lambda$ in Proposition \ref{reduction} has been defined for $|\lambda|\le\delta$ and $\|z\|\ge R$, we may (and do) extend it continuously to $[-\delta,\delta]\times Z$. Since $w'(\lambda,\infty)=0$ (see (ii) of Proposition \ref{reduction}) and asymptotic linearity coincides with $H$-asymptotic linearity on $Z$ (because $\dim Z<\infty$), we have, setting $K_\lambda(z) := (I-P)N(w(\lambda,z)+z)$ and using Remark \ref{pot}, 
\be \label{asympt}
\lim_{\|z\|\to\infty} \frac{\|K_\lambda(z)\|}{\|z\|} = 0.
\ee
Suppose there is no asymptotic bifurcation at $\lambda_0=0$. Taking smaller $\delta$ and larger $R$ if necessary, $F_\lambda(z)\ne 0$ for any $|\lambda|\le\delta$ and $\|z\|\ge R$. Therefore the Brouwer degree $\deg(F_\lambda, B_R(0),0)$ (see e.g. \cite[Section 3.1]{am}) is well defined and independent of $\lambda\in [-\delta,\delta]$. Since $\delta < \dist(0,\sigma(L)\setminus\{0\})$, $L_{\pm\delta}$ are invertible. It follows therefore from \eqref{asympt} that if $R_0\ge R$ is sufficiently large, then $L_{\pm\delta}z-tK_{\pm\delta}(z)\ne 0$ for any $\|z\|\ge R_0$, $t\in[0,1]$. Hence by the excision property and the homotopy invariance of degree,
\[
k = \deg(F_{\pm\delta},B_R(0),0) = \deg(F_{\pm\delta},B_{R_0}(0),0) = \deg(L_{\pm\delta}|_Z, B_{R_0},0)
\]
for some $k\in\z$. Let $d_1,d_2$ be the number of negative eigenvalues (counted with their multiplicity) of respectively $L_\delta|_Z$ and $L_{-\delta}|_Z$. Then $k=(-1)^{d_1}=(-1)^{d_2}$ \cite[Lemma 3.3]{am}. However, since $d_1 = d_2 + \dim N(L)$ and $\dim N(L)$ is odd, this is impossible. So we have reached a contradiction to the assumption that there is no bifurcation. 

It remains to prove that there exists a bifurcating continuum. Usually this is done by first making the inversion $u\mapsto u/\|u\|^2$ and then showing there is a continuum bifurcating from 0 \cite{ra2, st2, to}. Here we give a slightly different argument avoiding inversion. 
 Let
\[
\Sigma := \{(z,\lambda)\in (Z\setminus B_R(0)) \times [-\delta,\delta]: F_\lambda(z)=0\}.
\]
Compactify $Z$ by adding the point at infinity and let $A:=\ol B_R(0)\times[-\delta,\delta]$, $B:=\{(\infty,0)\}$, $Y:=A\cup\Sigma\cup B$. Then $Y$ is compact, $A$ and $B$ are closed disjoint. We claim that if $R$ is large enough, there is a connected set $\Gamma\subset\Sigma$ such that $\{(\infty,0)\}\in\ol\Gamma$ (the closure taken in $Y$) and $\ol\Gamma\cap A\neq\emptyset$. Otherwise there exist $U$ and $V$ as in Lemma \ref{whyburn}. Since $U$ is compact and bounded, there exists a bounded open set $\mathcal{O}\subset Z\times[-\delta,\delta]$ such that $U\subset \mathcal{O}$ and $\partial \mathcal{O}\cap \Sigma=\emptyset$. Letting $\mathcal{O}_\lambda:= \{z: (z,\lambda)\in \mathcal{O}\}$ for $\lambda\in [-\delta,\delta]$, it follows from the excision property and the generalized version of the homotopy invariance property of degree \cite[Theorem 4.1]{am} that $\deg(F_{\delta}, \mathcal{O}_{\delta},0) = \deg(F_{-\delta}, \mathcal{O}_{-\delta},0)$, a contradiction since by the same argument as above $\deg(F_{\delta}, \mathcal{O}_{\delta},0) = (-1)^{k_1}$, $\deg(F_{-\delta}, \mathcal{O}_{-\delta},0)=(-1)^{k_2}$ and $k_1,k_2$ have different parity.
\end{proof}

In the proof of Theorem \ref{thm1a} we shall use Gromoll-Meyer theory. Below we summarize some pertinent facts which are special cases of much more general results of \cite{ks} where functionals were considered in a Hilbert space $E$ with filtration, i.e., with a sequence $(E_n)$ of subspaces such that $E_n\subset E_{n+1}$ for all $n$ and $\bigcup_{n=1}^\infty E_n$ is dense in $E$. In the terminology of \cite{ks}, here we have the trivial filtration (i.e., $Z_n=Z$ for all $n$) which, together with the fact that $\dim Z<\infty$,  considerably simplifies the proofs. An alternative approach is via the Conley index theory, see e.g. \cite{be, da}, in particular \cite[Corollary 2.3]{be} and \cite[Theorem 2]{da}. 

Let $\vp: Z\to \r$ be a function such that $\nabla\vp$ is locally Lipschitz continuous. Suppose also $K = K(\vp) := \{z\in Z: \nabla\vp(z)=0\}$ is bounded. A pair $(\bbw,\bbw^-)$ of closed subsets of $Z$ will be called \emph{admissible} (for $\vp$ and $K$) if 
\begin{enumerate}
\item[(i)] $K\subset \text{int}(\bbw)$ and $\bbw^-\subset\partial \bbw$;
\item[(ii)] $\vp|_\bbw$ is bounded;
\item[(iii)] There exist a locally Lipschitz continuous vector field $V$ defined in a neighbourhood $N$ of $\bbw$ and a continuous function $\beta: N\to \r^+$ such that $\|V(z)\|\le 1$, $\la V(z),\vp(z)\ra \ge \beta(z)$ for all $z\in N$, and $\beta$ is bounded away from 0 on compact subsets of $N\setminus K$ (we shall call $V$ \emph{admissible} for $(\bbw,\bbw^-)$);
\item[(iv)] $\bbw^-$ is a piecewise $C^1$-manifold of codimension 1, $V$ is transversal to $\bbw^-$, the flow $\eta$ of $-V$ can leave $\bbw$ only via $\bbw^-$ and if $z\in \bbw^-$, then $\eta(t,z)\notin \bbw$ for any $t>0$.  
\end{enumerate}
Let $H^*$ denote the \v Cech (or Alexander-Spanier) cohomology with coefficients in $\z_2$ and let the critical groups $c^*(\vp,K)$ of the pair $(\vp,K)$ be defined by
\[
c^*(\vp,K) := H^*(\bbw,\bbw^-).
\]

\begin{lemma} \label{summary}
Suppose $\vp$ satisfies (PS). \\
(i) For each $R>0$ there exists a bounded admissible pair $(\bbw,\bbw^-)$ for $\vp$ and $K$ such that $B_R(0)\subset \bbw$. \\
(ii) If $(\bbw_1,\bbw_1^-)$ and $(\bbw_2,\bbw_2^-)$ are two admissible pairs for $\vp$ and $K$, then  $H^*(\bbw_1,\bbw_1^-) \cong H^*(\bbw_2,\bbw_2^-)$ (i.e., $c^*(\vp,K)$ is well defined). \\
(iii) Suppose $\{\vp_\lambda\}_{\lambda\in[0,1]}$ is a family of functions satisfying (PS) and such that $\nabla\vp_\lambda$ is locally Lipschitz continuous, $\lambda\mapsto \nabla\vp_\lambda$ is continuous, uniformly on bounded subsets of $Z$, and $K(\vp_\lambda)\subset B_R(0)$ for some $R>0$ and all $\lambda\in[0,1]$. Then $c^*(\vp_\lambda,K(\vp_\lambda))$ is independent of $\lambda$. 
\end{lemma}

This lemma corresponds to Lemma 2.13 and Propositions 2.12, 2.14 in \cite{ks}. Note that condition (PS)$^*$ there is in our setting (i.e. for trivial filtration) equivalent to (PS). 

\begin{proof}[Outline of proof]
(i) Choose $R,a,b$ so that $K\subset B_R(0)$, $a<\vp(z)<b$ for all $z\in B_R(0)$ and let 
\be \label{vf}
V(z) := \frac{\nabla\vp(z)}{1+\|\nabla\vp(z)\|}. 
\ee
Clearly, the flow $\eta$ given by
\[
\frac{d\eta}{dt} = -V(\eta), \quad \eta(0,z)=z
\]
is defined on $\r\times Z$. Let 
\[
\bbw := \{\eta(t,z): t\ge 0,\ z\in B_R(0),\ \vp(\eta(t,z))\ge a\}, \quad \bbw^-:= \bbw\cap\vp^{-1}(a). 
\]
Then $(\bbw,\bbw^-)$ is an admissible pair. The proof follows that of \cite[Lemma 2.13]{ks} but is simpler - there is no need for using cutoff functions. Note that (here and below) the Palais-Smale condition rules out the possibility that $\vp(\eta(t,z))>a$ and $\|\eta(t,z)\|\to\infty$ as $t\to\infty$, hence $t\mapsto \eta(t,z)$ either approaches $K$ as $t\to\infty$ or hits $\bbw^-=\vp^{-1}(a)$ in finite time.

(ii) Assume that $\vp$ is unbounded below and above (the other cases are simpler but somewhat different). Let $(\bbw_0,\bbw_0^-)$ be an admissible pair and $V_0$ a corresponding admissible vector field. As $\vp|_{\bbw_0}$ is bounded, we may choose $a,b$ so that $a<\vp(z)<b$ for all $z\in\bbw_0$. Since $(\bbw_1,\bbw_1^-) := (\vp^{-1}([a,b]),\vp^{-1}(a))$ is an admissible pair, it suffices to show that $H^*(\bbw_0,\bbw_0^-)\cong H^*(\bbw_1,\bbw_1^-)$. Put $V(z) := \chi_0(z)V_0(z)+\chi_1(z)V_1(z)$, where $V_1$ is given by \eqref{vf} and $\{\chi_0,\chi_1\}$ is a Lipschitz continuous partition of unity such that $\chi_0(z)=1$ on $\bbw_0$ and $\chi_1(z)=1$ in a neighbourhood of $\partial \bbw_1$. Denote the flow of $-V$ by $\eta$. Let $A := \{\eta(t,z): t\ge 0,\ z\in \bbw_0^-\}\cap \bbw_1$ and $\bbw = \bbw_0\cup A$, $\bbw^- := \bbw\cap \bbw_1^-$. Then $(\bbw,\bbw^-)$ is a an admissible pair and using $\eta$ one obtains a strong deformation retraction of $A$ onto $\bbw^-$. So $H^*(A,\bbw^-)=0$ and by exactness of the cohomology sequence of the triple $(\bbw,A,\bbw^-)$ and the strong excision property we have $H^*(\bbw,\bbw^-)\cong H^*(\bbw,A) \cong H^*(\bbw_0,\bbw_0^-)$. We also have, by excision again, $H^*(\bbw,\bbw^-)\cong H^*(\bbw\cup \bbw_1^-,\bbw_1^-)$. Finally, using the flow $\eta$ once more, we obtain a deformation of $(\bbw_1,\bbw_1^-)$ into $(\bbw\cup \bbw_1^-,\bbw_1^-)$ which leaves $\bbw\cup \bbw_1^-$ and $\bbw_1^-$ invariant. Hence $(\bbw\cup \bbw_1^-,\bbw_1^-)$ and $(\bbw_1,\bbw_1^-)$ are homotopy equivalent and thus have the same cohomology. Putting everything together gives $H^*(\bbw_0,\bbw_0^-)\cong H^*(\bbw_1,\bbw_1^-)$. More details of the proof may be found in \cite[Propositions 2.12 and 2.7]{ks}. 

(iii) Let $\lambda_0\in[0,1]$. It suffices to show that $c^*(\vp_\lambda,K(\vp_\lambda))$ is constant for $\lambda$ in a neighbourhood of $\lambda_0$. Denote the vector field for $\vp_\lambda$ given as in \eqref{vf} by $V_\lambda$ and choose an admissible pair $(\bbw_{\lambda_0},\bbw_{\lambda_0}^-)$ for $\vp_{\lambda_0}$ and $K(\vp_{\lambda_0})$ such that $B_{R_1}(0)\subset \bbw_{\lambda_0}$, where $R_1>R$. By the construction in (i), we may assume $V_{\lambda_0}$ is admissible for this pair. Let $\wt V(z) := \chi_1(z)V_\lambda(z)+\chi_2(z)V_{\lambda_0}(z)$, where $\{\chi_1,\chi_2\}$ is a partition of unity subordinate to the sets $B_{R_1}(0)$ and $\bbw_{\lambda_0}\setminus \ol B_R(0)$. 
It is easy to see that if $|\lambda-\lambda_0|$ is small enough, then $(\bbw_{\lambda_0},\bbw_{\lambda_0}^-)$ is an admissible pair for $\vp_\lambda$, $K(\vp_\lambda)$ and $\wt V$ is a corresponding admissible field. Note in particular that 
\[
\|\nabla\vp_\lambda(z)\| \ge \|\nabla\vp_{\lambda_0}(z)\| - \|\nabla\vp_\lambda(z)-\nabla\vp_{\lambda_0}(z)\|>0
\]
for $z\in \bbw_{\lambda_0}\setminus \ol B_R(0)$, so indeed $\wt V$ is admissible. Hence $c^*(\vp_\lambda,K(\vp_\lambda)) \cong c^*(\vp_{\lambda_0},K(\vp_{\lambda_0}))$. 
\end{proof}

\begin{proof}[Proof of Theorem \ref{thm1a}]
Let $\vp_\lambda$ be given by \eqref{redfcl} and extend it to the whole space $Z$ according to Remark \ref{remphi}. If $\lambda_0=0$ is not an asymptotic bifurcation point for \eqref{oe}, then it follows from Proposition \ref{varred} that $\nabla\vp_\lambda(z)\ne 0$ for $\lambda\in[-\delta,\delta]$ and $\|z\|>R$, possibly after choosing a smaller $\delta$ and larger $R$. By assumption, $\vp_0$ satisfies (PS) and since $L_\lambda$ has bounded inverse if $0<|\lambda|\le\delta$, we see using \eqref{asympt} that $\nabla\vp_\lambda$ is bounded away from 0 as $\|z\|\to\infty$. Hence all $\vp_\lambda$, $|\lambda|\le\delta$, satisfy (PS). By Lemma \ref{summary}, $c^*(\vp_\lambda, K(\vp_\lambda))$ is independent of $\lambda\in[-\delta,\delta]$. For $\lambda=\delta$, let $Z=Z_\delta^+\oplus Z_\delta^-$ and $z=z^++z^-\in Z_\delta^+\oplus Z_\delta^-$, where $Z_\delta^\pm$ are the maximal $L_\delta$-invariant subspaces of $Z$ on which $L_\delta$ is respectively positive and negative definite.  Choose $\eps>0$ such that $\la \pm L_\delta z^\pm,z^\pm \ra \ge \eps\|z^\pm\|^2$ and let 
\[
\bbw := \{z\in Z: \|z^+\|\le R_0,\ \|z^-\|\le R_0\}, \quad \bbw^- := \{z\in \bbw: \|z^-\|=R_0\}.
\]
Recall $K_\lambda(z) = (I-P)N(w(\lambda,z)+z)$. Taking a sufficiently large $R_0$,
\[
\la \nabla\vp_\delta(z),z^+\ra = \la L_\delta z,z^+\ra - \la K_\delta(z),z^+\ra \ge \eps\|z^+\|^2-\frac14\eps \|z\|\,\|z^+\| > 0, \quad z\in \bbw,\ \|z^+\|=R_0.
\]
Similarly,
\[
\la \nabla\vp_\delta(z),z^-\ra < 0, \quad z\in \bbw,\ z^-\in \bbw^-.
\]
So the flow of $-\nabla\vp_\delta$ is transversal to $\bbw^-$ and can leave $\bbw$ only via $\bbw^-$. Hence $(\bbw,\bbw^-)$ is an admissible pair for $\vp_\delta$ and $K(\vp_\delta)$, and $V=\nabla\vp_\delta$ is a corresponding admissible vector field. Note that this pair is also admissible for the quadratic functional $\Psi_\delta(z) := \frac12\la L_\delta z,z\ra$. Since 0 is the only critical point of $\Psi_\delta$, it follows e.g. from \cite[Corollary 8.3]{mw} that if $m$ is the Morse index of $\Psi_\delta$, then 
\[
c^q(\vp_\delta, K(\vp_\delta)) = c^q(\Psi_\delta,0) = \delta_{q,m} \z_2. 
\]
A similar argument shows that $c^q(\vp_{-\delta}, K(\vp_{-\delta})) = \delta_{q,n}\z_2$, where $n$ is the Morse index of $\Psi_{-\delta}$. As the Morse index changes (by $\dim N(L)$) when $\lambda$ passes through 0, $m\ne n$ and $c^*(\vp_\delta, K(\vp_\delta))\ne c^*(\vp_{-\delta}, K(\vp_{-\delta}))$. This is the desired contradiction. 
\end{proof}

\section{Proofs of Theorems \ref{thm2} and \ref{thm2a}} \label{sec:proof2}

We assume throughout this section that $V\in L^\infty(\rn)$ and $f$ satisfies $(f_1)$-$(f_3)$. We consider equation \eqref{se} which we  re-write in the form
\be \label{se2}
-\Delta u + V_0(x)u = \lambda u + g(x,u), \quad x\in\rn,
\ee
where we have put $V_0(x):=V(x)-m(x)$ and $g(x,u):=f(x,u)-m(x)u$. Let $\lambda_0$ be an isolated eigenvalue of finite multiplicity for $-\Delta+V_0$. Replacing $V_0(x)$ by $V_0(x)-\lambda_0$ we may assume without loss of generality that $\lambda_0=0$. 

Let $E := L^2(\rn)$ and $Lu := -\Delta u+V_0(x)u$. As we have pointed ot in the introduction, $L$ is a selfadjoint operator whose domain is the Sobolev space $H^2(\rn)$ and the graph norm of $L$ is equivalent to the Sobolev norm. (A brief argument: using the Fourier transform one readily sees that $-\Delta + 1: H^2(\rn) \to L^2(\rn)$ is an isomorphism; hence the conclusion follows because $V\in L^\infty(\rn)$.)

We define the operator $N$ (the Nemytskii operator) by setting
$$N(u):=g(\cdot,u(\cdot)),\quad u\in E.$$
It follows from $(f_1)$ and Krasnoselskii's theorem \cite[Theorems 2.1 and 2.3]{kr} that  $N:E\to E$ is well defined and continuous. Let
$$G(x,s):=\int_0^sg(x,\xi)\,d\xi,\quad x\in\R^N,\ s\in\R$$
and
$$\psi(u):=\irn G(x,u)\,dx,\quad u\in  E.$$
Then $\psi\in C^1(E,\R)$ and
$$\nabla\psi(u)=N(u),$$
see \cite[Lemma 5.1]{kr}. Furthermore, let 
\[
\Phi_\lambda(u) := \frac12\la Lu-\lambda u,u\ra - \psi(u),\quad u\in X:= H^2(\rn).
\]
Then $\Phi_\lambda\in C^1(X,\r)$ and $\Phi_\lambda'(u)=0$ if and only if $u$ is a solution of \eqref{se2}. 

\begin{proof}[Proof of Theorem \ref{thm2}.]
We verify the assumptions of Theorem \ref{thm1}. First we show that $N$ is $H$-asymptotically linear and $N'(\infty)=0$. Let $u_n\to u$ and $\|t_nu_n\|\to\infty$ in $E$. Assume passing to a subsequence that $u_n\to u$ a.e. Since
\[
\frac{g(x,t_nu_n)^2}{t_n^2 } \le \left(\frac{\alpha(x)}{t_n} + (\beta+\|m\|_\infty)|u_n|\right)^2
\]
and $g(x,s)/s\to 0$ as $|s|\to\infty$, it follows from the Lebesgue dominated convergence theorem that
\[
\lim_{n\to\infty} \frac{\|N(u_n)\|^2}{t_n^2} = \irn \lim_{n\to\infty}\frac{g(x,t_nu_n)^2}{(t_nu_n)^2}u_n^2\,dx = 0.  
\]
Hence (i) of Theorem \ref{thm1} is satisfied. 
Since $\Lip_\infty(N) = \Lip_\infty(g) \le \Lip(g) < \dist(0,\sigma_e(L))$ (where the second inequality follows by assumption), also (ii) of this theorem holds. This completes the proof.
\end{proof}

\begin{remark} \label{nobif}
\emph{
As we have mentioned in the introduction, the condition $\Lip(g) < \dist(\lambda_0, \sigma_e(L))$ is sharp in the sense that there may be no asymptotic bifurcation if $\Lip(g) > \dist(\lambda_0, \sigma_e(L))$ and other assumptions of Theorem \ref{thm2} are satisfied. Let $N=1$ and suppose $V_0\in C^1(\r)$, $V_0'(x)\le 0$ for $x$ large, $\lim_{|x|\to\infty}V_0(x) = V_0(\infty)$ exists and $ \inf\{\la Lu,u\ra: \|u\|_2=1\} < V_0(\infty)$. Then $\sigma_e(L) = [V_0(\infty),\infty)$ and $\lambda_0 := \inf\sigma(L) < \inf\sigma_e(L)$ is a simple eigenvalue. Assume without loss of generality that $\lambda_0=0$. Assume also that $g$ is independent of $x$, of class $C^1$, $g(0)=\lim_{|s|\to\infty}g(s)/s=0$ and $\xi := V_0(\infty) + g'(0) < 0$. Given $\eps>0$, we may choose $g$ so that $\Lip(g)=-g'(0)\in (V_0(\infty), V_0(\infty)+\eps)$. So 
\[
\Lip(g)-\eps < \dist(0,\sigma_e(L)) = V_0(\infty) < \Lip(g), 
\]
and according to \cite[Theorem 5.4]{st3} and the remarks following it, there is no asymptotic bifurcation at any $\lambda>\xi$, in particular, not at $\lambda_0=0$. See also the explicit Example 1 after the proof of Theorem 5.4 in \cite{st3}. A similar conclusion holds for $N\ge 2$, see \cite[Theorem 5.6]{st3}. 
}
\end{remark}

In the proof of Theorem \ref{thm2a} we shall need an auxiliary result.  Let $\lambda_0=0$ and write $w(z)=w(0,z)$. Then $w(z)$ satisfies equation \eqref{P}, i.e. we have
\[
Lw(z) = PN(w(z)+z).
\]

\begin{lemma} \label{bound}
Suppose $(f_1)$-$(f_4)$ are satisfied. Then $\|w(z)\|_\infty \le C$ for some constant $C>0$ and all $\|z\| > R$. 
\end{lemma}

\begin{proof}
Recall $L := -\Delta+V_0$, where $L: D(L)\subset L^2(\rn)\to L^2(\rn)$. We also define $\wt L:= -\Delta+V_0$ when $-\Delta+V_0$ is regarded as an operator in $L^\infty(\rn)$ (i.e.,  $\wt L: D(\wt L)\subset L^\infty(\rn) \to L^\infty(\rn)$). By \cite[Theorem]{hv}, $\sigma(L)=\sigma(\wt L)$ and isolated eigenvalues of $L$ and $\wt L$ have the same multiplicity. Since $Z$ is spanned by eigenfunctions of $-\Delta +V_0$ corresponding to isolated eigenvalues and such eigenfunctions decay exponentially \cite[Theorem C.3.4]{Simon}, $Z\subset L^2(\rn)\cap L^\infty(\rn)$. It follows therefore from \cite[Theorem III.6.17]{Kato} that  there is an $L$-invariant decomposition $L^\infty(\rn) = \wt Z\oplus\wt W$, where $\wt Z=Z$. Moreover, by \cite[(III.6.19)]{Kato},
\[
Q := I-P = -\frac1{2\pi i}\int_\gamma (L-\lambda I)^{-1}\,d\lambda,
\]
where $\gamma$ is a smooth simple closed curve (in $\mathbb{C}$) which encloses all eigenvalues corresponding to $Z$ and no other points in $\sigma(L)$. By \cite[Proposition 2.1]{hv}, $(L-\lambda I)^{-1}|_{L^2(\rn)\cap L^\infty(\rn)} = (\wt L-\lambda I)^{-1}|_{L^2(\rn)\cap L^\infty(\rn)}$. Hence $Q|_{L^2(\rn)\cap L^\infty(\rn)} = \wt Q|_{L^2(\rn)\cap L^\infty(\rn)} $, where $\wt Q$ denotes the $\wt L$-inva\-riant projection of $L^\infty(\rn)$ on $Z$, and the same equality holds for $P$ and $\wt P:=I-\wt Q$.
$\wt P$ is a projection on a subspace of finite codimension, hence it is continuous and therefore $(f_1)$, $(f_4)$ imply $y=y(z):=PN(w(z)+z)\in L^2(\rn)\cap L^\infty(\rn)$ and $\|y\|_\infty\le C_1$ for some $C_1$ independent of $z\in Z\setminus \ol B_R(0)$. Since $L|_{\wt W}$ has bounded inverse, $\|\wt w\|_\infty\le C$, where  $\wt w=\wt w(z):=\wt L^{-1}y$ (note that for $w=w(z)=L^{-1}y$ we only have a $z$-dependent $L^2$-bound because $N(w+z)$ is not uniformly bounded in $L^2(\rn)$).  
 
 We complete the proof by showing that $w=\wt w$. Let $\mu_n\notin\sigma(L)$, $\mu_n\to 0$. By the resolvent equation  \cite[(I.5.5) and \S III.6.1]{Kato},
 \[
w = L^{-1}y = (L-\mu_nI)^{-1}y  - \mu_n L^{-1}(L-\mu_nI)^{-1}y
 \] 
 and
 \[
\wt w = \wt L^{-1}y = (\wt L-\mu_nI)^{-1}y  - \mu_n \wt L^{-1}(\wt L-\mu_nI)^{-1}y.
 \] 
Let $v_n:=(L-\mu_nI)^{-1}y$. Since $y\in L^2(\rn)\cap L^\infty(\rn)$,  \cite[Proposition 2.1]{hv} implies  $v_n=(\wt L-\mu_nI)^{-1}y$ as well. As the last term on each of the right-hand sides above tends to 0, $(v_n)$ is a Cauchy sequence in $L^2(\rn)\cap L^\infty(\rn)$ (with the norm $\|\cdot\|_{L^2\cap L^\infty} := \|\cdot\|_2+\|\cdot\|_\infty$) which yields $w=\wt w$.  
\end{proof}

\begin{proof}[Proof of Theorem \ref{thm2a}.]
We have already shown that assumptions (i)-(iii) of Theorem \ref{thm1a} are satisfied. Suppose first that $(f_4)$ and $(f_5)$ hold. We only need to verify that $\vp_0$ satisfies (PS).  
 Recall from \eqref{redfcl} that for $\|z\|>R$
\[
\vp_0(z) = \Phi_0(w(z)+z),
\]
where we have put $w(z)=w(0,z)$, and by Proposition \ref{varred}, we have
\be \label{gradphi}
\la \nabla\vp_0(z),\zeta\ra = \la Lz,\zeta\ra - \irn g(x,w(z)+z)\zeta\,dx \quad \text{for all } z,\zeta\in Z,\ \|z\|>R.
\ee
Let $z=z^++z^-+z^0 \in Z^+\oplus Z^-\oplus Z^0$, where $Z^+, Z^-$ respectively denote the subspaces of $Z$ corresponding to the positive and the negative part of the spectrum of $L|_Z$ and $Z^0:=N(L)\subset Z$. 
Let $(z_n)\subset Z$ be such that $\nabla\vp_0(z_n)\to 0$. It suffices to consider $z_n$ with $\|z_n\|>R$, and we shall show that $(z_n)$ is bounded.
Since $Z$ is spanned by eigenfunctions of $-\Delta +V_0$ and $\dim Z<\infty$, it follows from \cite[Theorem C.3.4]{Simon} that there are constants $\delta,C_0>0$ such that $|z(x)|\leq C_0e^{-\delta|x|}$ for all $x\in\R^N$ and all $z\in Z$ with $\|z\|\le 1$. In particular, such $z$ are uniformly bounded in $L^p(\R^N)$ for any $p\in[1,\infty]$. Using this, $(f_4)$ and equivalence of norms in $Z$, we obtain
\[ 
\left| \la Lz_n^+,z\ra \right| \le \left|\irn g(x,w(z_n)+z_n)z\,dx \right| + o(1)\|z\| \le c_1\|z\| \le c_2 \quad \text{for all } z\in Z^+,\ \|z\|=1.
\]
Hence $(z_n^+)$ is bounded and a similar argument shows that so is $(z_n^-)$. Suppose $\|z_n^0\| \to \infty$ and write $z_n^0 = t_n\wt z_n^0$, where $\|\wt z_n^0\|=1$. Passing to a subsequence, $\wt z_n^0\to\wt z^0\in Z^0$. Denote
\[
v_n := w(z_n)+z_n^++z_n^-.
\]
 We shall obtain a contradiction with the assumption  $\nabla\vp_0(z_n)\to 0$ by showing that
\be \label{notto}
\la -\nabla\vp_0(z_n),\wt z_n^0 \ra = \irn g(x,v_n+t_n\wt z_n^0)\wt z_n^0\,dx \not\to 0.
\ee
By Lemma \ref{bound}, the sequence $(w(z_n))$ is bounded in $L^\infty(\rn)$, and since so are the sequences $(z_n^\pm)$, $v_n(x)+t_n\wt z_n^0(x)\to\pm\infty$ for all $x\in A_\pm:=\{x\in\R^N:\pm\wt z^0(x)>0\}$.

Suppose $\pm g_\pm\ge 0$. Since $g$ is bounded and $\wt z_n^0$ is uniformly bounded in $L^1(\rn)$, we may use the Lebesgue dominated convergence theorem to obtain
\[
\lim_{n\to\infty} \int_{A_\pm}g(x,v_n+t_n\wt z_n^0)\wt z_n^0\,dx = \int_{A_\pm} g_\pm\wt z^0\,dx \ge 0.
\]
By the unique continuation property \cite[Proposition 3 and Remark 2]{defg}, $\wt z^0(x)\ne 0$ a.e. Hence the measure of $\rn\setminus (A_+\cup A_-)$ is 0 and thus
\be \label{lal}
\int_{A_+} g_+ \wt z^0\,dx + \int_{A_-} g_-  \wt z^0\,dx > 0.
\ee
This implies \eqref{notto}. If $\pm g_\pm\le 0$, the same argument remains valid after making some obvious changes.

Suppose now that $(f_4)$ and $(f_6)$ are satisfied. Here we do not know whether (PS) holds for $\vp_0$, however, we will construct an admissible pair directly by adapting an argument in \cite{ko}, see in particular the proof of Theorem 4.5 there. Suppose $g(x,s)s\ge 0$ in $(f_6)$ and let
\[
\bbw := \{z\in Z: \|z^\pm\|\le R_0,\ \|z^0\|\le R_1\}, \quad \bbw^- := \{z\in \bbw: \|z^-\|=R_0 \text{ or } \|z^0\|=R_1\}
\]
($R_0, R_1$ to be determined). Boundedness of $g$ and equivalence of norms in $Z$ yield
\[
\left|\irn g(x,w(z)+z)z^+\,dx\right| \le c_3\|z^+\|. 
\]
Since $\la \pm Lz,z^\pm\ra \ge \eps\|z^\pm\|^2$ for some $\eps>0$, $\la \nabla\vp_0(z),z^+\ra \ge \eps \|z^+\|^2-c_3\|z^+\|>0$ if $\|z^+\|=R_0$ and $\la \nabla\vp_0(z),z^-\ra < 0$ if $\|z^-\|=R_0$ provided $R_0$ is large enough. We want to show that there exists a (large) $R_1$ such that $\la \nabla\vp_0(z),z^0\ra < 0$ for $z$ with $\|z^-\|=R_0$ and $\|z^0\|=R_1$. Assuming the contrary, $\liminf_{n\to\infty} \la \nabla\vp_0(z_n),z_n^0 \ra\ge 0$ for a sequence $(z_n)$ such that $\|z_n^0\|\to\infty$. Below we use the same notation as in \eqref{notto}. We have 
\[
0 = \la -\nabla\vp_0(z_n),w(z_n)\ra = \irn g(x,v_n+t_n\wt z^0_n)w(z_n)\,dx,
\]
 $g(x,s)\to 0$ as $|s|\to\infty$ (because $h_\pm\in L^\infty(\rn)$ by $(f_6)$) and $|g(x,v_n+t_n\wt z_n^0)z_n^\pm| \le c_4 e^{-\delta|x|}$. So  according to the Lebesgue dominated convergence theorem, 
\[
\lim_{n\to\infty} \irn g(x,v_n+t_n\wt z_n^0)v_n\, dx = 0. 
\]
Hence Fatou's lemma and $(f_6)$ give
\[
\liminf_{n\to\infty} \int_{A_\pm}g(x,v_n+t_n\wt z_n^0)t_n\wt z_n^0\,dx =  \liminf_{n\to\infty} \int _{A_\pm}  g(x,v_n+t_n\wt z_n^0)(v_n+t_n\wt z_n^0)\,dx
\ge \int_{A_\pm} h_\pm \,dx \ge 0.
\]
Since by assumption at least one of the integrals on the right-hand side is positive (possibly infinite), 
\[
\liminf_{n\to\infty} \la -\nabla\vp_0(z_n),z_n^0 \ra = \liminf_{n\to\infty} \irn g(x,v_n+t_n\wt z_n^0)t_n\wt z_n^0\,dx > 0,
\]
a contradiction. So $R_1$ exists as required and $(\bbw,\bbw^-)$ is an admissible pair. Now it is easy to see as in the proof of (iii) of Lemma \ref{summary} that this is also an admissible pair for $\vp_{\pm\delta}$ if $\delta$ is small enough. As in the proof of Theorem \ref{thm1a} one shows that the critical groups for $\vp_\delta$ and $\vp_{-\delta}$ are different, and this forces  bifurcation. 

If $g(x,s)s\le 0$, a similar argument shows that $\la \nabla\vp_0(z),z^0\ra > 0$ for some $R_1$, hence the exit set for the flow is $\bbw^- := \{z\in \bbw: \|z^-\|=R_0\}$. 
\end{proof}

\begin{remark}
\emph{
Note that \eqref{lal} is a variant of the Landesman-Lazer condition introduced in \cite{ll} and Theorem \ref{thm2a} remains valid if one assumes \eqref{lal} holds for all $z\in N(L)$. This is slightly less restrictive than $(f_5)$. The reason that we have chosen $(f_5)$ is that it is a general condition on $f$, with no reference to eigenfunctions corresponding to $\lambda_0$. $(f_6)$ is a kind of strong resonance condition because $g(x,s)\to 0$ as $|s|\to\infty$. Note also that our arguments show that under the assumptions of Theorem \ref{thm2a} there is a uniform bound for solutions of \eqref{se} with $\lambda=\lambda_0$. 
}
\end{remark}

\noindent\textbf{Acknowledgements.} We would like to thank Charles Stuart for pointing out the references \cite{dh, to2} to us.

\end{document}